\documentclass[a4paper,12pt]{article}
\usepackage{amsmath, amsfonts, amssymb, amsthm}
\usepackage[english]{babel}

\newcounter{num}[section]

\newenvironment{theorem}
{\refstepcounter{num}%
\bigskip\noindent\nopagebreak[4]{\bf Theorem~\arabic{section}.\arabic{num}. }\it}

\newenvironment{proposition}
{\refstepcounter{num}%
\bigskip\noindent\nopagebreak[4]{\bf Proposition~\arabic{section}.\arabic{num}. }\it}

\newenvironment{lemma}
{\refstepcounter{num}%
\bigskip\noindent\nopagebreak[4]{\bf Lemma~\arabic{section}.\arabic{num}. }\it}

\newenvironment{remark}
{\refstepcounter{num}%
\bigskip\noindent\nopagebreak[4]{\bf Remark~\arabic{section}.\arabic{num}. }}

\newcommand{\Ss}{{\mathbf{S}}}
\newcommand{\V}{{\mathrm{V}}}
\newcommand{\Var}{{\mathrm{Var}}}
\newcommand{\Irr}{{\mathrm{Irr}}}

\newcommand{\lb}{{\langle}}
\newcommand{\rb}{{\rangle}}

\newcommand{\pr}{{\prime}}

\newcommand{\s}{{\sigma}}

\newcommand{\stir}[2]{\genfrac{\{}{\}}{0pt}{}{#1}{#2}}

\begin{document}
\title{On irreducible algebraic sets over linearly ordered semilattices}
\author{Artem N. Shevlyakov}

\maketitle

\abstract{Equations over linearly ordered semilattices are studied. For any equation $t(X)=s(X)$ we find irreducible components of its solution set and compute the average number of irreducible components of all equations in $n$ variables.}
\section{Introduction}

This paper is devoted to the following problem. One can define a notion of an equation over a linearly ordered semilattice $L_l=\{a_1,a_2,\ldots,a_l\}$ (the formal definition of an equation is given below). A set $Y$ is {\it algebraic} if it is the solution set of some system of equations over $L_l$.  Let us consider an equation $t(X)=s(X)$ over $L_l$, and $Y$ be the solution set of $t(X)=s(X)$. One can find algebraic sets $Y_1,Y_2,\ldots,Y_m$ such that $Y=\bigcup_{i=1}^m Y_i$. One can decompose each $Y_i$ into a union of other algebraic sets, etc. This process terminates after a finite number of steps and gives a decomposition of $Y$ into a union of {\it irreducible} algebraic sets $Y_i$ (the sets $Y_i$ are called the {\it irreducible components }of $Y$). Roughly speaking, irreducible algebraic sets are ``atoms'' which form any algebraic set. The size and the number of such ``atoms'' are important characteristics of the semilattices $L_l$, since there are connections between irreducible algebraic sets and universal theory of linearly ordered semilattices (see \cite{uni_Th_II}). Moreover, the number of irreducible components was involved in the estimation of lower bounds of algorithm complexity (see~\cite{ben-or} for more details).

In this paper (Section~\ref{sec:decomposition_properties}) we study the properties of the irreducible components of the solution set $Y$ of an equation $t(X)=s(X)$. Precisely, we prove that the union of irreducible algebraic sets $Y=\bigcup_{i=1}^m Y_i$ is redundant, i.e. the intersections $\bigcap_{i\in I}Y_i$ ($|I|<m$) consists of many points (Proposition~\ref{pr:redundant}). Moreover, for any equation $t(X)=s(X)$ in $n$ variables we count the number $m$ of irreducible components (see~(\ref{eq:Irr(k_1,k_2,n,l)})), and in Section~\ref{sec:average} we count the average number $\overline{\Irr}(n,l)$ of irreducible components of the solution sets of equations in $n$ variables.

\section{Main definitions}

Let $L_l=\{a_1,a_2,\ldots,a_l\}$ be the linearly ordered semilattice of $l$ elements and $a_1<a_2<\ldots <a_l$. The multiplication in $L_l$ is defined by  $a_i\cdot a_j=a_{\min(i,j)}$. Obviously, the linear order on $L_l$ can be expressed by the multiplication as follows
\[
a_i\leq a_j\Leftrightarrow a_ia_j=a_i.
\]
A {\it term} $t(X)$ in variables $X=\{x_1,x_2,\ldots,x_n\}$ is a commutative word in letters $x_i$.

Let $\Var(t)$ be the set of all variables occurring in a term $t(X)$.
Following~\cite{uni_Th_II}, an {\it equation} is an equality of terms $t(X)=s(X)$. Below we consider inequalities $t(X)\leq s(X)$ as equations, since $t(X)\leq s(X)$ is the short form of $t(X)s(X)=t(X)$. Notice that we consider equations as {\it ordered pairs} of terms, i.e. the expressions $t(X)=s(X)$, $s(X)=t(X)$ are {\it different} equations. Let $Eq(n)$ denote the set of all equations in $X=\{x_1,x_2,\ldots,x_n\}$ variables (we assume that each $t(X)=s(X)\in Eq(n)$ contains the occurrences of all variables $x_1,x_2,\ldots,x_n$). An equation $t(X)=s(X)\in Eq(n)$ is said to be a {\it $(k_1,k_2)$-equation} if $|\Var(t)\setminus\Var(s)|=k_1$ and $|\Var(s)\setminus\Var(t)|=k_2$. For example, $x_1x_2=x_1x_3x_4$ is a $(1,2)$-equation. Let $Eq(k_1,k_2,n)\subseteq Eq(n)$ be the set of all $(k_1,k_2)$-equations in $n$ variables. Obviously,
\begin{equation}
Eq(n)=\bigcup_{(k_1,k_2)\in K_n}Eq(k_1,k_2,n),
\label{eq:Eq(n)}
\end{equation}
where 
\[
K_n=\{(k_1,k_2)\mid k_1+k_2\leq n\}\setminus\{(0,n),(n,0)\}.
\]

Each equation $t(X)=s(X)\in Eq(k_1,k_2,n)$ is uniquely defined by $k_1$ variables in the left part and by $k_2$ other variables in the right part (the residuary $n-k_1-k_2$ variables should occur in both parts of the equation). Thus, 
\[
\#Eq(k_1,k_2,n)=\binom{n}{k_1}\binom{n-k_1}{k_2}.
\]
By~(\ref{eq:Eq(n)}), one can compute
\[
\#Eq(n)=3^n-2.
\]  

\begin{remark}
In this paper we consider only equations $t(X)=s(X)$ with $n>l$, i.e. the number of variables occurring in $t(X)=s(X)$ is more than the order of the semilattice $L_l$. The case $n\leq l$ needs the different technic and was announced in~\cite{malov}.
\end{remark}

\bigskip

A point $P\in L_l^n$ is a {\it solution} of an equation $t(X)=s(X)$ if $t(P),s(P)$ define the same element in the semilattice $L_l$. By the properties of linearly ordered semilattices, a point $P=(p_1,p_2,\ldots,p_n)$ is a solution of $t(X)=s(X)$ iff there exist variables $x_i\in\Var(t)$, $x_j\in\Var(s)$ such that $p_i=p_j$ and $p_i\leq p_k$ for all $1\leq k\leq n$.  The set of all solutions of an equation $t(X)=s(X)$ is denoted by $\V(t(X)=s(X))$.

An arbitrary set of equations is called a {\it system}. The set of all solutions $\V(\Ss)$ of a system $\Ss=\{t_i(X)=s_i(X)\mid i\in I\}$ is defined as $\bigcap_{i\in I}\V(t_i(X)=s_i(X))$. A set $Y\subseteq L_l^n$ is called {\it algebraic over }$L_l$ if there exists a system $\Ss$ in $n$ variables with $\V(\Ss)=Y$. An algebraic set $Y$ is {\it irreducible} if $Y$ is not a proper finite union of other algebraic sets.  

\begin{proposition}
Any algebraic set $Y$ over $L_l$ is a finite union of irreducible sets
\begin{equation}
Y=Y_1\cup Y_2\cup \ldots\cup Y_m,\quad Y_i\nsubseteq Y_j \mbox{ for all $i\neq j$},
\label{eq:union_Y_general}
\end{equation}
and this decomposition is unique up to a permutation of components.
\end{proposition} 
\begin{proof}
A semilattice $S$ is {\it equationally Noetherian} if for any infinite system $\Ss$ in variables $X=\{x_1,x_2,\ldots,x_n\}$ there exists a finite subsystem $\Ss^\pr\subseteq \Ss$ with the same solution set. According to~\cite{uni_Th_II}, the decomposition~(\ref{eq:union_Y_general}) holds for any algebraic set $Y$ over an equationally Noetherian semilattice $S$. Thus, it is sufficient to prove that $L_l$ is equationally Noetherian.

However the condition $|Eq(n)|<\infty$ gives that there is not any infinite system over $L_l$. Thus, $L_l$ is equationally Noetherian.
\end{proof}

The subsets $Y_i$ from the union~(\ref{eq:union_Y_general}) are called the {\it irreducible components} of $Y$.

Let $Y$ be an algebraic set over $L_l$ defined by a system $\Ss(X)$. One can define an equivalence relation $\sim_Y$ over the set of all terms in variables $X$ as follows
\[
t(X)\sim_Y s(X)\Leftrightarrow t(P)=s(P) \mbox{ for any point $P\in Y$}.
\] 
The set of $\sim_Y$-equivalence classes is called {\it the coordinate semilattice of $Y$} and denoted by $\Gamma(Y)$ (see~\cite{uni_Th_II} for more details). The following statement describes the coordinate semilattices of irreducible algebraic sets.

\begin{proposition}
A set $Y$ is irreducible over $L_l$ iff $\Gamma(Y)$ is embedded into $L_l$ 
\label{pr:gamma_is_embedded_for_irr}
\end{proposition}  
\begin{proof}
Following~\cite{uni_Th_II}, $\Gamma(Y)$ is discriminated by $L_l$ iff $Y$ is irreducible (see~\cite{uni_Th_II} for the definition of the discrimination). However for a finite semilattice $L_l$ the discrimination is equivalent to the embedding.
\end{proof}

There are different algebraic sets over $L_l$ with isomorphic coordinate semilattices. Such sets are called {\it isomorphic}. For example, the following sets
\[
Y_1=\V(\{x_1\leq x_2\leq x_3\}),\; Y_2=\V(\{x_3\leq x_2\leq x_1\})
\]
has the isomorphic coordinate semilattices
\[
\Gamma(Y_1)=\lb x_1,x_2,x_3\mid x_1\leq x_2\leq x_3\rb\cong L_3,
\]
\[
\Gamma(Y_2)=\lb x_1,x_2,x_3\mid x_3\leq x_2\leq x_1\rb\cong L_3.
\]
Thus, $Y_1,Y_2$ are isomorphic.

\newpage
\section{Example}

Let $n=3$, $l=2$. We have exactly $Eq(3)=3^3-2=25$ equations in three variables over $L_2$. The following table contains the  information about such equations over $L_2$. The second column contains systems which define irreducible components of the solution set of an equation in the first column. A cell of the table contains $\uparrow$ if an information in this cell is similar to the cell above.

\begin{tabular}{|c|c|c|}
\hline
Equations&Irreducible components (IC)&Number of IC\\
\hline
$x_1x_2x_3=x_1x_2x_3$&$x_1\leq x_2=x_3 \cup x_1=x_2\leq x_3\cup$ &$6$\\
&$x_2\leq x_1=x_3 \cup x_3\leq x_1=x_2\cup$&\\
&$x_1=x_3\leq x_2 \cup x_2=x_3\leq x_1$&\\
\hline
$x_1=x_1x_2x_3$,&$x_1\leq x_2=x_3\cup x_1=x_2\leq x_3\cup$&$3$\\
$x_1x_2x_3=x_1$&$x_1=x_3\leq x_2$&\\
\hline
$x_2=x_1x_2x_3$,&$\uparrow$&$3$\\
$x_1x_2x_3=x_2$&&\\
\hline
$x_3=x_1x_2x_3$,&$\uparrow$&$3$\\
$x_1x_2x_3=x_3$&&\\
\hline
$x_1=x_2x_3$,&$x_1=x_2\leq x_3\cup x_1=x_3\leq x_2$&$2$\\
$x_2x_3=x_1$&&\\
\hline
$x_2=x_1x_3$,&$\uparrow$&$2$\\
$x_1x_3=x_2$&&\\
\hline
$x_3=x_1x_2$,&$\uparrow$&$2$\\
$x_1x_2=x_3$&&\\
\hline
$x_1x_2=x_1x_3$,&$x_1=x_2\leq x_3\cup x_1=x_3\leq x_2\cup$&$4$\\
$x_1x_3=x_1x_2$&$x_1\leq x_2= x_3\cup x_2=x_3\leq x_1$&\\
\hline
$x_1x_2=x_2x_3$,&$\uparrow$&$4$\\
$x_2x_3=x_1x_2$&&\\
\hline
$x_1x_3=x_2x_3$,&$\uparrow$&$4$\\
$x_2x_3=x_1x_3$&&\\
\hline
$x_1x_2=x_1x_2x_3$,&$x_1=x_2\leq x_3\cup x_1=x_3\leq x_2\cup$&$5$\\
$x_1x_2x_3=x_1x_2$&$x_1\leq x_2= x_3\cup x_2=x_3\leq x_1\cup$&\\
&$x_2\leq x_1= x_3$&\\
\hline
$x_1x_3=x_1x_2x_3$,&$\uparrow$&$5$\\
$x_1x_2x_3=x_1x_3$&&\\
\hline
$x_2x_3=x_1x_2x_3$,&$\uparrow$&$5$\\
$x_1x_2x_3=x_2x_3$&&\\
\hline
\end{tabular} 

One can directly compute the average number of irreducible components of algebraic sets defined by equations in three variables:
\begin{equation}
\overline{\Irr}(3,2)=\frac{6+2(3+3+3+2+2+2+4+4+4+5+5+5)}{25}=\frac{90}{25}=3.6
\label{eq:Irr(3,2)_handy}
\end{equation}

Recall that in Section~\ref{sec:average} we obtain the general expression for $\overline{\Irr}(n,l)$~(\ref{eq:Irr}). Clearly,~(\ref{eq:Irr}) gives~(\ref{eq:Irr(3,2)_handy}) for $n=3$, $l=2$ (see the proof in~(\ref{eq:Irr_n_2}) and~(\ref{eq:Irr_3_2_from_formula})).

\section{Decompositions of algebraic sets}
\label{sec:decomposition_properties}
Let $Y$ denote the solution set of an equation $t(X)=s(X)$ over the semilattice $L_l=\{a_1,a_2,\ldots,a_l\}$. The table above shows that any irreducible component divides the variables $X$ into $l$ classes and sorts the classes in some order. The following definition formalizes such properties of irreducible components. 

A disjoint partition $\sigma=(X_1,X_2,\ldots,X_l)$ of the set $X=\{x_1,x_2,\ldots,x_n\}$ is called \textit{ordered} if there is a linear order $\leq_\sigma$ on $\sigma$: $X_1\leq_\sigma X_2\leq_\sigma\ldots\leq_\sigma X_l$. Let $\chi_\s(x_i)$ denote the class $X_k$ with $x_i\in X_k$.

We shall denote $x_i=_\s x_j$ ($x_i\leq_\s x_j$) if $\chi(x_i)=\chi(x_j)$ (respectively, $\chi_\s(x_i)\leq_\s\chi_\s(x_j)$). 

An ordered partition $\sigma$ is $Y$-\textit{irreducible} if the set $X_1$ (the minimal set of the order $\leq_\sigma$) contains a variable from $t(X)$ and a variable from $s(X)$.  

For example, an equation $x_1x_2x_3=x_1$ over $L_2$ has the following $Y$-irreducible partitions: $(\{x_1\},\{x_2,x_3\})$, $(\{x_1,x_2\},\{x_3\})$, $(\{x_1,x_3\},\{x_2\})$. Such partitions obviously correspond to irreducible components of $\V(x_1x_2x_3=x_1)$ in the table above. 

Any $Y$-irreducible partition $\sigma$ defines an algebraic set $Y_\sigma$ as follows
\[
Y_\sigma=\V(\Ss_\sigma)=\V(\bigcup_{x_i=_\s x_j}\{x_i=x_j\}\bigcup_{x_i<_\s x_j}\{x_i\leq x_j\}).
\]
For example, the partition $\s=(\{x_2,x_3\},\{x_1\})$  defines the system
\[
\Ss_\s=\{x_2=x_3,x_2\leq x_1,x_3\leq x_1\}.
\]
for $Y=\V(\{x_1x_2=x_1x_3\})$.

\begin{lemma}
The set $Y_\sigma$ defined by a $Y$-irreducible partition $\sigma$ is an irreducible algebraic set, and moreover $\Gamma(Y_\s)\cong L_l$.  
\label{l:Y_sigma_is_irreducible}
\end{lemma}
\begin{proof}
By the definition of a coordinate semilattice, $\Gamma(Y_\sigma)$ is generated by the elements $\{x_1,x_2,\ldots,x_n\}$ and has the following defined relations
\[
\{x_i=x_j\mid \mbox{ if $x_i=_\s x_j$}\}\cup
\{x_i\leq x_j\mid \mbox{ if $x_i\leq_\s x_j$}\}.
\]
It is easy to see that all elements $x_i$ are linearly ordered in $\Gamma(Y_\s)$.   
Thus, $\Gamma(Y_\sigma)$ is a linearly ordered semilattice, and it is isomorphic to $L_l$. By Proposition~\ref{pr:gamma_is_embedded_for_irr}, the set $Y_\sigma$ is irreducible.  
\end{proof}

The following lemma gives the decomposition of the set $Y=\V(t(X)=s(X))$ via ordered partitions. 

\begin{lemma}
The set $Y=\V(t(X)=s(X))$ is a union
\begin{equation}
\label{eq:union_of_Y}
Y=\bigcup_{\mbox{$\sigma$ is $Y$-irreducible}}Y_\sigma
\end{equation}
\label{l:union_of_Y}
\end{lemma}
\begin{proof}
Let $P=(p_1,p_2,\ldots,p_n)\in Y$. One can define an equivalence relation $\sim_P$ as follows
\[
x_i\sim_P x_j\Leftrightarrow p_i=p_j.
\] 
Thus, we obtain equivalence classes $\{X_1^P,X_2^P,\ldots,X_k^P\}$. Since $p_i\in L_l$, $k\leq l$. One can define a linear order $x_i\leq_P x_j$ if $p_i\leq p_j$. The order $\leq_P$ induces a linear order over the classes $\{X_i\}$. Let us fix a pair of variables $x_t,x_s\in X_1^P$ (probably, $x_t,x_s$ is the same variable) such that $x_t\in\Var(t)$ and $x_s\in\Var(s)$ (such pair $(x_t,x_s)$ always exists, since $P$ satisfies the equation $t(X)=s(X)$). Let us find a set $Y_\s$ with $P\in Y_\s$ by the following procedure.

{\bf Procedure}

Input: a set of $k$ equivalence classes $\s_0=(X_1^P,X_2^P,\ldots,X_k^P)$ with the linear order $\leq_P$.

Output: $\s=(X_1,X_2,\ldots,X_l)$ with a linear order $\leq_\s$.

Step 0: Put $\s=\sigma_0$. If $l=k$ terminate the procedure, otherwise go to the step 1.

Step $j$ ($1\leq j\leq l-k$): 
\begin{enumerate}
\item Take an arbitrary equivalence class $X_i\in\sigma=(X_1,X_2,\ldots,X_{k+j-1})$ such that $|X_i|\geq 2$ and $X_i$ contains a variable $x\in X\setminus\{x_t,x_s\}$. Such class always exists, since $n>l>k+j-1$.
\item Move $x$ from $X_i$ to a new class $X^\pr$ and define a linear order $\leq_\s$ by $X_i\leq_\s X^\pr\leq X_{i+1}$. Put $\s=(X_1,X_2,\ldots,X_{i},X^\pr,X_{i+1},\ldots,X_{l+j-1})$. Go to the next step.
\end{enumerate}

Roughly speaking, the procedure increases the number of classes preserving the relation $<_\s$.

After the procedure we obtain an ordered partition $\s$ of $l$ equivalence classes $X_i$. The procedure does not move the variables $x_t,x_s$, therefore $x_t,x_s\in X_1$ and $\sigma$ is a $Y$-irreducible partition.

Let us prove $P\in Y_\s=\V(\Ss_\s)$. An equation $x_i\leq x_j\in \Ss_\s$ (one can similarly consider an equality $x_i=x_j\in\Ss_\s$) is not satisfied by $P$ if $p_i> p_j$ or equivalently $x_j<_P x_i$. Since the procedure preserves the relation $<_\s$, we have  $x_j<_\s x_i$, and by the definition of $\Ss_\s$, the equation $x_i\leq x_j$ can not occur in $\Ss_\s$. Thus, we came to the contradiction.

\bigskip

Let us prove now $Y_\s\subseteq Y$ for each $\s$. Consider a point $P=(p_1,p_2,\ldots,p_n)\in Y_\s$. Since $\s=(X_1,X_2,\ldots,X_l)$ is a $Y$-irreducible partition, the class $X_1$ contains variables $x_t\in\Var(t)$, $x_s\in\Var(s)$ and $p_t=p_s$. Since $X_1$ is the minimal class of the order $\leq_\s$, 
\[
x_t\leq x_i\in\Ss_\s, \; x_s\leq x_i\in\Ss_\s \mbox{ for any }i\in [1,n]\setminus\{t,s\}.
\]
Thus, $p_t=p_s\leq p_i$ for any $1\leq i\leq n$, and we have 
\[
t(P)=p_t=p_s=s(P)\Rightarrow P\in\V(t(X)=s(X))=Y.
\]  

\end{proof}

Let $\sigma=(X_1,X_2,\ldots,X_l)$ be a $Y$-irreducible partition of $X$. Let us define a point $P_\sigma=(p_1,p_2,\ldots,p_n)\in L_l^n$ by
\[
p_i=a_k\mbox{ if $x_i\in X_k$}. 
\]

\begin{lemma}
The point $P_\sigma$ belongs to the set $Y_\sigma$, and $P_\sigma\notin Y_{\sigma^\pr}$ for each $Y$-irreducible partition $\sigma^\pr\neq \sigma$. Thus, in the union~(\ref{eq:union_of_Y}) $Y_{\sigma}\nsubseteq Y_{\sigma^\pr}$ for distinct partitions $\sigma,\sigma^\pr$.
\label{l:about_point_P_sigma}
\end{lemma}
\begin{proof}
One can directly prove that $P_\sigma\in\V(\Ss_\sigma)=Y_\s$. 

Let us take an irreducible partition 
\[
\s^\pr=(X_1^\pr,X_2^\pr,\ldots,X_l^\pr)\neq\s=(X_1,X_2,\ldots,X_l).
\]
There exist variables $x_i,x_j$ such that $x_i<_{\s} x_j$ but $x_i\geq_{\s^\pr} x_j$. For the point $P_\s$ we have $p_i<p_j$, therefore $P_\s$ does not satisfy the equation $x_i\geq x_j\in \Ss_{\s^\pr}$, and $P_\s\notin Y_{\s^\pr}$.  
\end{proof}

According to Lemmas~\ref{l:Y_sigma_is_irreducible},~\ref{l:union_of_Y},~\ref{l:about_point_P_sigma}, we obtain the following statement.

\begin{theorem}
The number of $Y$-irreducible partitions of a set $Y=\V(t(X)=s(X))$ is equal to the number of irreducible components of $Y$.  
\label{th:number_of_irr_compionents}
\end{theorem}

\bigskip

The next statement describes the properties the union~(\ref{eq:union_of_Y}).

\begin{proposition}
Let~(\ref{eq:union_of_Y}) be a union of the irreducible components of a set $Y=\V(t(X)=s(X))$ over $L_l$. Then
\begin{enumerate}
\item a point $P$ belongs to all $Y_\s$ iff $P=(a,a,,\ldots,a)$ for some $a\in L_l$;
\item 
\[
Y_\s\setminus\bigcup_{\s^\pr\neq\s}Y_{\s^\pr}=\{P_\s\}
\]
(it follows that the decomposition~(\ref{eq:union_of_Y}) is redundant, i.e. each point of $Y\setminus\bigcup_\s\{P_\s\}$ is covered by at least two irreducible components);
\item all irreducible components are isomorphic to each other;

\item $|Y_\s|=\binom{2l-1}{l}$ for each $\s$.

\end{enumerate}
\label{pr:redundant}
\end{proposition}
\begin{proof}
\begin{enumerate}
\item Obviously, $P=(a,a,\ldots,a)$ satisfies all systems $\Ss_\s$, so $P\in\bigcap_\s Y_\s$.

Let us consider a point $Q=(q_1,q_2,\ldots,q_n)$ with $q_i<q_j$. It is clear that $Q$ does not satisfy any set $Y_\s$ with $x_i\geq_\s x_j$. Thus, $Q\notin\bigcap_\s Y_\s$.

\item In Lemma~\ref{l:about_point_P_sigma} we proved $P_\s\in Y_\s$. By the definition, only the point $P_\s$ makes all inequalities $\leq$ of the system $\Ss_\s$ strict. Thus, for any point $P=(p_1,p_2,\ldots,p_n)\in Y_\s\setminus\{ P_\s\}$ there exists an equation $x_i\leq x_j\in\Ss_\s$ such that $p_i=p_j$.  Below we find an irreducible partition ${\s^\pr}$ with $P\in Y_{\s^\pr}$. 

Let $\s=(X_1,X_2,\ldots,X_l)$, $x_i\in X_{i^\pr}$ and without loss of generality one can assume that $x_j\in X_{i^\pr+1}$. If $i^\pr\neq 1$ we put $\s^\pr=(X_1^\pr,X_2^\pr,\ldots,X_l^\pr)$ where
\begin{equation}
X_k^\pr=\begin{cases}\
X_k\mbox{ if $k\neq i^\pr$, $k\neq i^\pr+1$},\\
(X_{i^\pr+1}\setminus\{x_{j}\})\cup\{x_i\}\mbox{ if $k=i^\pr+1$},\\
(X_{i^\pr}\setminus\{x_{i}\})\cup\{x_j\}\mbox{ if $k=i^\pr$}
\end{cases}
\label{eqq:ast}
\end{equation}
Since $X_1^\pr=X_1$, $\s^\pr$ is a $Y$-irreducible partition. The system $\Ss_{\s^\pr}$ contains $x_j\leq x_i$ instead of $x_i\leq x_j\in\Ss_\s$. Since other relations in the systems $\Ss_{\s^\pr},\Ss_\s$
are the same, $P\in \V(\Ss_{\s^\pr})=Y_{\s^\pr}$.

Suppose now $i^\pr=1$. Without loss of generality we assume $x_i\in\Var(t)$. By the definition of a $Y$-irreducible partition, there exists a variable $x_k\in X_1\cap \Var(s)$. If $x_j\in\Var(t)$ we can define $\s^\pr$ by~(\ref{eqq:ast}). In this case $X_1^\pr$ contains variables $x_j\in\Var(t)$, $x_k\in\Var(s)$, so $\s^\pr$ is an $Y$-irreducible partition and $P\in Y_{\s^\pr}$. Otherwise ($x_j\in\Var(s)$), one can take $x_k$ instead $x_i$ and repeat all reasonings above.

\item The statement immediately follows from Lemma~\ref{l:Y_sigma_is_irreducible}.

\item For $\s=(X_1,X_2,\ldots,X_l)$ the number $|Y_\s|$ is equal to the number of sequences $X_1\leq X_2\leq\ldots\leq X_l$ with $X_i\in\{a_1,a_2,\ldots,a_l\}$. According to combinatorics, the number of such monotone sequences is $\binom{2l-1}{l}$. 
\end{enumerate}
\end{proof}

\section{Average number of irreducible components}
\label{sec:average}

Let $\stir{n}{m}$ be the Stirling number of the second kind. By the definition, $\stir{n}{m}$ is the number of all partitions of an $n$-element set into $m$ non-empty unlabelled subsets. The number $\stir{n}{m}^\ast=m!\stir{n}{m}$ obviously equals the number of all partitions of $n$-element set into $m$ {\it labelled} non-empty subsets. Thus, there are exactly $\stir{n}{l}^\ast$ ordered partitions $\s=(X_1,X_2,\ldots,X_l)$ of the set of variables $X$, $|X|=n$ into $l$ equivalence classes. {\it An ordered partition $\s=(X_1,X_2,\ldots,X_l)$ is not $Y$-irreducible if either $X_1\subseteq \Var(t)\setminus\Var(s)$ or  $X_1\subseteq \Var(s)\setminus\Var(t)$} For a $(k_1,k_2)$-equation $t(X)=s(X)$ there exists
\[
\sum_{i=1}^{k_1}\binom{k_1}{i}\stir{n-i}{l-1}^\ast
\]
partitions $\s$ with $X_1\subseteq \Var(t)\setminus\Var(s)$. Similarly, there exist 
\[
\sum_{i=1}^{k_2}\binom{k_2}{i}\stir{n-i}{l-1}^\ast
\]
partitions $\s$ with $X_1\subseteq \Var(s)\setminus\Var(t)$. 

By Theorem~\ref{th:number_of_irr_compionents}, for a $(k_1,k_2)$-equation $t(X)=s(X)$ the number of irreducible components ($Y$-irreducible partitions) equals  
\begin{equation}
\Irr(k_1,k_2,n,l)=\stir{n}{l}^\ast-\sum_{i=1}^{k_1}\binom{k_1}{i}\stir{n-i}{l-1}^\ast-\sum_{i=1}^{k_2}\binom{k_2}{i}\stir{n-i}{l-1}^\ast.
\label{eq:Irr(k_1,k_2,n,l)}
\end{equation}
The average number of irreducible components of algebraic sets defined by equations from $Eq(n)$ is
\begin{multline*}
\overline{\Irr}(n,l)=\frac{\sum_{(k_1,k_2)\in K_n}\#Eq(k_1,k_2,n)\Irr(k_1,k_2,n,l)}{\#Eq(n)}=\\
\frac{\sum_{k_1=0}^{n-1}\sum_{k_2=0}^{n-k_1}\#Eq(k_1,k_2,n)\Irr(k_1,k_2,n,l)-\#Eq(0,n,n)\Irr(0,n,n,l)}{\#Eq(n)}
\end{multline*} 
Below we compute $\overline{\Irr}$ using the following denotations:
\begin{enumerate}
\item $A\stackrel{(1)}{=}B$: an expression $B$ is obtained from $A$ by the binomial theorem
\[
(a+b)^n=\sum_{i=0}^n\binom{n}{i}a^ib^{n-i}.
\]
\item $A\stackrel{(2)}{=}B$: an expression $B$ is obtained from $A$ by the following identity of binomial coefficients
\[
\binom{a}{b}\binom{b}{c}=\binom{a}{c}\binom{a-c}{b-c}.
\] 
\item $A\stackrel{(3)}{=}B$: an expression $B$ is obtained from $A$ by the recurrence relation of Stirling numbers
\[
\stir{a+1}{b}=b\stir{a}{b}+\stir{a}{b-1}.
\]
\item $A\stackrel{(4)}{=}B$: an expression $B$ is obtained from $A$ by the following identity of Stirling numbers
\[
\stir{a+1}{b+1}=\sum_{i=0}^a\binom{a}{i}\stir{i}{b}.
\]
Remark that in the last formula one can change the sum $\sum_{i=0}^a$ to $\sum_{i=c}^a$ ($c<b$), since $\stir{c}{b}=0$ for $c<b$.
\end{enumerate}

We have
\begin{multline*}
\#Eq(0,n,n)\Irr(0,n,n,l)=\binom{n}{0}\binom{n}{n}\left(\stir{n}{l}^\ast-\sum_{i=1}^{n}\binom{n}{i}\stir{n-i}{l-1}^\ast    \right)=\\
\stir{n}{l}^\ast-\sum_{i=1}^{n}\binom{n}{n-i}\stir{n-i}{l-1}^\ast=
\stir{n}{l}^\ast-\sum_{j=0}^{n-1}\binom{n}{j}\stir{j}{l-1}^\ast=
\stir{n}{l}^\ast-(l-1)!\sum_{j=0}^{n-1}\binom{n}{j}\stir{j}{l-1}\stackrel{(4)}{=}\\
\stir{n}{l}^\ast-(l-1)!\left(\stir{n+1}{l}-\stir{n}{l-1}\right)\stackrel{(3)}{=}\stir{n}{l}^\ast-(l-1)!l\stir{n}{l}=0,
\end{multline*} 

\begin{multline*}
\sum_{k_1=0}^{n-1}\sum_{k_2=0}^{n-k_1}\#Eq(k_1,k_2,n)\Irr(k_1,k_2,n)=\\
\sum_{k_1=0}^{n-1}\sum_{k_2=0}^{n-k_1}\binom{n}{k_1}\binom{n-k_1}{k_2}
\left(\stir{n}{l}^\ast-\sum_{i=1}^{k_1}\binom{k_1}{i}\stir{n-i}{l-1}^\ast-\sum_{i=1}^{k_2}\binom{k_2}{i}\stir{n-i}{l-1}^\ast\right)=\\
\sum_{k_1=0}^{n-1}\sum_{k_2=0}^{n-k_1}\binom{n}{k_1}\binom{n-k_1}{k_2}\stir{n}{l}^\ast-
\sum_{k_1=0}^{n-1}\sum_{k_2=0}^{n-k_1}\binom{n}{k_1}\binom{n-k_1}{k_2}\sum_{i=1}^{k_1}\binom{k_1}{i}\stir{n-i}{l-1}^\ast-\\
\sum_{k_1=0}^{n-1}\sum_{k_2=0}^{n-k_1}\binom{n}{k_1}\binom{n-k_1}{k_2}\sum_{i=1}^{k_2}\binom{k_2}{i}\stir{n-i}{l-1}^\ast=S_1-S_2-S_3,
\end{multline*} 
where
\[
S_1=\stir{n}{l}^\ast\sum_{k_1=0}^{n-1}\binom{n}{k_1}2^{n-k_1}\stackrel{(1)}{=}\stir{n}{l}^\ast(3^n-1),
\]
\begin{multline*}
S_2\stackrel{(2)}{=}
\sum_{k_1=0}^{n-1}\sum_{i=1}^{k_1}\binom{n}{k_1}\binom{k_1}{i}\stir{n-i}{l-1}^\ast\sum_{k_2=0}^{n-k_1}\binom{n-k_1}{k_2}\stackrel{(1)}{=}\\
\sum_{k_1=0}^{n-1}\sum_{i=1}^{k_1}\binom{n}{i}\binom{n-i}{k_1-i}\stir{n-i}{l-1}^\ast 2^{n-k_1}=
\sum_{i=1}^{n-1}\binom{n}{i}\stir{n-i}{l-1}^\ast\sum_{k_1=i}^{n-1}\binom{n-i}{k_1-i} 2^{n-k_1}=\\
\sum_{i=1}^{n-1}\binom{n}{i}\stir{n-i}{l-1}^\ast\sum_{j=0}^{n-i-1}\binom{n-i}{j} 2^{n-i-j}=
\sum_{i=1}^{n-1}\binom{n}{i}\stir{n-i}{l-1}^\ast\left(\sum_{j=0}^{n-i}\binom{n-i}{n-i-j} 2^{n-i-j} -1\right)\stackrel{(1)}{=}\\
\sum_{i=1}^{n-1}\binom{n}{i}\stir{n-i}{l-1}^\ast\left(3^{n-i} -1\right).
\end{multline*} 
Computing
\begin{multline*}
\sum_{i=1}^{n-1}\binom{n}{i}\stir{n-i}{l-1}^\ast=(l-1)!\sum_{j=1}^{n-1}\binom{n}{j}\stir{j}{l-1}\stackrel{(4)}{=}
(l-1)!\left(\stir{n+1}{l}- \stir{n}{l-1}\right)\stackrel{(3)}{=}\\
(l-1)!l\stir{n}{l}=\stir{n}{l}^\ast,
\end{multline*} 
we obtain 
\[
S_2=\sum_{i=1}^{n-1}\binom{n}{i}\stir{n-i}{l-1}^\ast 3^{n-i}-\stir{n}{l}^\ast=S(n,l)-\stir{n}{l}^\ast,
\]
where 
\[
S(n,l)=\sum_{i=1}^{n-1}\binom{n}{i}\stir{n-i}{l-1}^\ast 3^{n-i}.
\]

Let us compute
\begin{multline*}
S_3=\sum_{k_1=0}^{n-1}\sum_{i=1}^{n-k_1}\sum_{k_2=i}^{n-k_1}\binom{n}{k_1}\binom{n-k_1}{i}\binom{n-k_1-i}{k_2-i}\stir{n-i}{l-1}^\ast=\\
\sum_{k_1=0}^{n-1}\sum_{i=1}^{n-k_1}\binom{n}{k_1}\binom{n-k_1}{i}\stir{n-i}{l-1}^\ast\sum_{k_2=i}^{n-k_1}\binom{n-k_1-i}{k_2-i}\stackrel{(1)}{=}
\sum_{k_1=0}^{n-1}\sum_{i=1}^{n-k_1}\binom{n}{k_1}\binom{n-k_1}{i}\stir{n-i}{l-1}^\ast 2^{n-k_1-i}\stackrel{(2)}{=}\\
\sum_{k_1=0}^{n-1}\sum_{i=1}^{n-k_1}\binom{n}{i}\binom{n-i}{n-k_1-i}\stir{n-i}{l-1}^\ast 2^{n-k_1-i}=
\sum_{i=1}^{n}\binom{n}{i}\stir{n-i}{l-1}^\ast 2^{n-i}\sum_{k_1=0}^{n-i}\binom{n-i}{k_1} 2^{-k_1}\stackrel{(1)}{=}\\
\sum_{i=1}^{n}\binom{n}{i}\stir{n-i}{l-1}^\ast 2^{n-i}\left(1+\frac{1}{2}\right)^{n-i}=
\sum_{i=1}^{n}\binom{n}{i}\stir{n-i}{l-1}^\ast 3^{n-i}=S(n,l)+\binom{n}{n}\stir{n-n}{l-1}^\ast=S(n,l).
\end{multline*}

Finally, we obtain
\begin{multline}
\label{eq:Irr}
\overline{\Irr}(n,l)=\frac{S_1-S_2-S_3-0}{3^n-2}=
\frac{\stir{n}{l}^\ast(3^n-1)-(S(n,l)-\stir{n}{l}^\ast)-S(n,l)}{3^n-2}=\\
\frac{3^n\stir{n}{l}^\ast-2S(n,l)}{3^n-2}.
\end{multline}

Let us compute $\overline{\Irr}(n,2)$ using the following identities of the Stirling numbers
\[
\stir{n}{1}=1,\; \stir{n}{2}=2^{n-1}-1.
\]

We have
\[
S(n,2)=\sum_{i=1}^{n-1}\binom{n}{i}\cdot 1\cdot 3^{n-i}=\sum_{i=1}^{n-1}\binom{n}{i}3^{n-i}\stackrel{(1)}{=}4^n-3^n-1,
\]
therefore
\begin{equation}
\overline{\Irr}(n,2)=\frac{3^n\cdot 2(2^{n-1}-1)-2(4^n-3^n-1)}{3^n-2}=\frac{6^n-2\cdot 4^n+2}{3^n-2}.
\label{eq:Irr_n_2}
\end{equation}
In particular, $n=3$ gives
\begin{equation}
\label{eq:Irr_3_2_from_formula}
\overline{\Irr}(3,2)=\frac{6^3-2\cdot 4^3+2}{3^3-2}=\frac{90}{25}=3.6
\end{equation}
that coincides with~(\ref{eq:Irr(3,2)_handy}).

The following statement gives the estimation of $\overline{\Irr}(n,l)$.

\begin{proposition}
The number $\overline{\Irr}(n,l)$ satisfies
\[
\frac{1}{3}\stir{n}{l}^\ast\leq\overline{\Irr}(n,l)\leq \stir{n}{l}^\ast
\]
\label{pr:Irr_double_ineq}
\end{proposition}
\begin{proof}

One can bound $S(n,l)$ as follows
\begin{multline*}
S(n,l)\leq 3^{n-1}\sum_{i=1}^{n-1}\binom{n}{j}\stir{j}{l-1}^\ast\stackrel{(4)}{=}3^{n-1}(l-1)!\left(\stir{n+1}{l}-\stir{n}{l-1}\right)\stackrel{(3)}{=}\\
3^{n-1}(l-1)!l\stir{n}{l}=3^{n-1}\stir{n}{l}^\ast,
\end{multline*}
and similarly
\[
S(n,l)\geq 3\sum_{i=1}^{n-1}\binom{n}{j}\stir{j}{l-1}^\ast=
3\stir{n}{l}^\ast.
\]
Thus,
\[
\overline{\Irr}(n,l)\leq\frac{3^n\stir{n}{l}^\ast-2\cdot 3\stir{n}{l}^\ast}{3^n-2}=\stir{n}{l}^\ast\frac{3^n-6}{3^n-2}\leq\stir{n}{l}^\ast,
\]
and
\[
\overline{\Irr}(n,l)\geq\frac{3^n\stir{n}{l}^\ast-2\cdot 3^{n-1}\stir{n}{l}^\ast}{3^n-2}=\stir{n}{l}^\ast\frac{3^n-2\cdot 3^{n-1}}{3^n-2}\geq
\stir{n}{l}^\ast\frac{3^n-2\cdot 3^{n-1}}{3^n}=\frac{1}{3}\stir{n}{l}^\ast.
\]

\end{proof}

\begin{proposition}
For a fixed $l$ and $n\to\infty$ we have the asymptotic equivalence
\[
\overline{\Irr}(n,l)\sim l^n.
\]
\end{proposition}
\begin{proof}
Using the following explicit formula for Stirling numbers
\[
\stir{n}{l}=\frac{1}{l!}\sum_{j=0}^l(-1)^{l-j}\binom{l}{j}j^n,
\]
we obtain $\stir{n}{l}\sim l^n$ for fixed $l$ and $n\to\infty$.
By Proposition~\ref{pr:Irr_double_ineq}, we have 
\[
\overline{\Irr}(n,l)\sim \stir{n}{l}^\ast=l!\stir{n}{l}\sim l!l^n\sim l^n.
\] 
\end{proof}

The information of the author:

Artem N. Shevlyakov

Sobolev Institute of Mathematics

644099 Russia, Omsk, Pevtsova st. 13

Phone: +7-3812-23-25-51.

e-mail: \texttt{a\_shevl@mail.ru}
\end{document}